\documentclass[letterpaper,11pt]{amsart}
\usepackage{latexsym,array,delarray,amsthm,amssymb,epsfig,setspace}

\theoremstyle{plain}
\newtheorem{thm}{Theorem}[section]
\newtheorem{lemma}[thm]{Lemma}
\newtheorem{prop}[thm]{Proposition}
\newtheorem{cor}[thm]{Corollary}

\newtheorem*{thm*}{Theorem}
\newtheorem*{lemma*}{Lemma}
\newtheorem*{prop*}{Proposition}
\newtheorem*{cor*}{Corollary}
\newtheorem*{conj*}{Conjecture}

\theoremstyle{definition}
\newtheorem{defn}[thm]{Definition}
\newtheorem{ex}[thm]{Example}

\newtheorem{alg}[thm]{Algorithm}

\theoremstyle{remark}

\newcommand{\calb}{\mathcal{B}}
\newcommand{\calc}{\mathcal{C}}

\newcommand{\cale}{\mathcal{E}}

\newcommand{\mast}{\mathrm{MAST}}

\newcommand{\ind}{\mbox{$\perp \kern-5.5pt \perp$}}

\newcommand{\E}{\mathrm{E}}

\begin{document}

\title[Expected size of the maximum agreement subtree for a given shape]{Bounds on the expected size of the maximum agreement subtree for a given tree shape}
\author{Pratik Misra, Seth Sullivant }

\address{Department of Mathematics
\\North Carolina State University, Raleigh, NC, USA\\}

\email{smsulli2@ncsu.edu, pmisra@ncsu.edu}

\begin{abstract}
    
We show that the expected size of the maximum agreement subtree of two $n$-leaf trees,
uniformly random among all trees with the shape, is  $\Theta(\sqrt{n})$. 
To derive the lower bound, we prove a global structural result on a decomposition of 
rooted binary trees into subgroups of leaves called \textit{blobs}. To obtain the upper bound,
we generalize a first moment argument from \cite{Bernstein2015} for random tree distributions
that are exchangeable and not necessarily sampling consistent.
\end{abstract}

\maketitle

\section{Introduction}
Rooted binary trees are used in evolutionary biology to represent 
the evolution of a set of species where the leaves denote the existing 
species and the internal nodes denote the unknown ancestors. Biologists 
believe that there exists a single tree which can describe the evolution of all  living species. 
The study of methods to reconstruct evolutionary trees from biological data is 
the area called phylogenetics \cite{Felsenstein2004,Semple2003}.
Different tree reconstruction methods, and different datasets on the same set of
species, can lead to the reconstruction of different trees. In such cases, it 
is important to measure the distance between different trees constructed.
There are various distances between trees that are used including
Robinson-Foulds distance, distances based on tree rearrangements, and
the geodesic distance.   This paper focuses on the   \textit{maximum agreement subtree}
as a measure of discrepancy between trees.  

If $T$ is a rooted binary tree with $n$ leaves leaf 
labeled by $[n] = \{1,2, \ldots, n\}$ and $S$ is a subset of $[n]$,
then the \emph{binary restriction tree} $T|_S$ is defined as the 
subtree of $T$ obtained after deleting all the leaves that are not in $S$ 
and suppressing the internal nodes of degree $2$. The new tree $T|_S$ is rooted at the
most recent common ancestor of the set $S$. 
If $T_1$ and $T_2$ are two trees leaf labeled by $X$, 
then a subset $S\subseteq X$ is said to be an agreement set of 
$T_1$ and $T_2$ if $T_1|_S=T_2|_S$. A \emph{maximum agreement subtree} is a subtree that 
is obtained from an agreement set of $T_1$ and $T_2$ and is of maximal size.
Figures \ref{fig:tree} and \ref{fig:mast} give an example of two trees
and a maximum agreement subtree.

\begin{figure} 
\begin{center}
  \resizebox{!}{4cm}{
\includegraphics{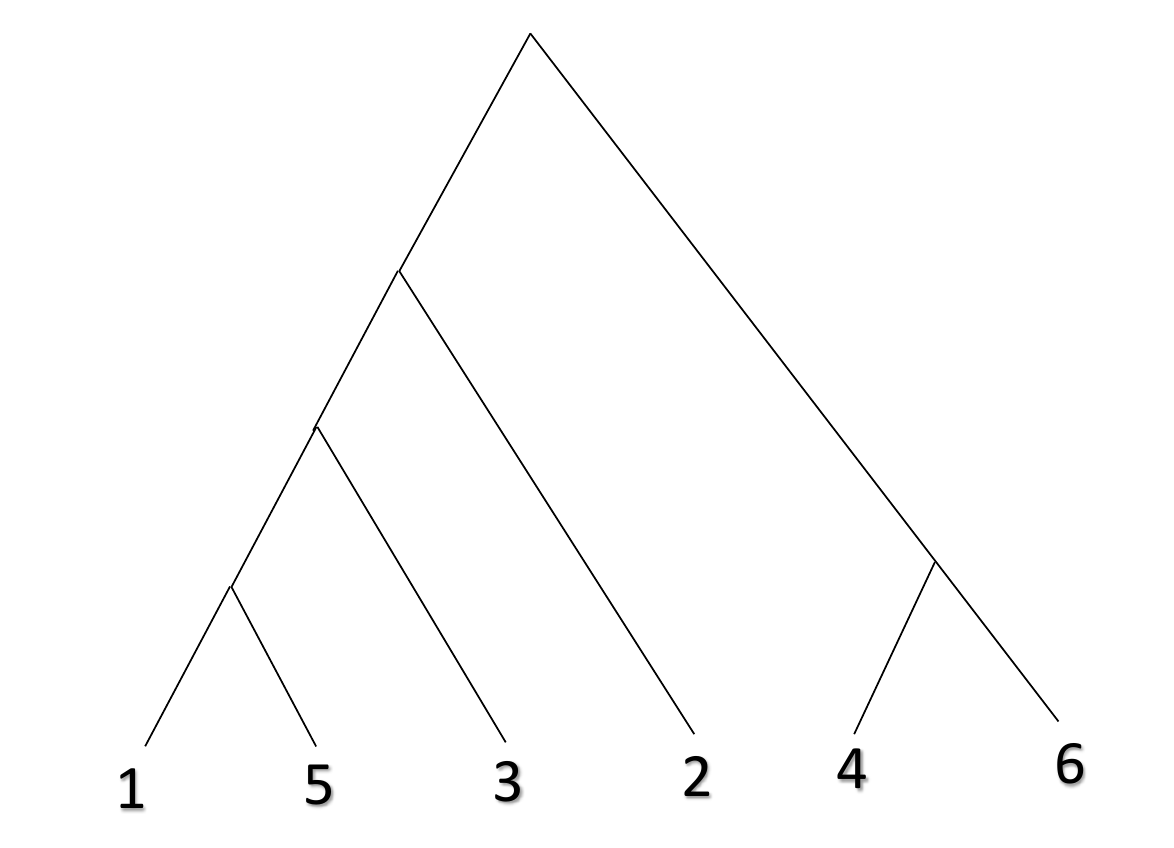}} \quad \quad   \resizebox{!}{4cm}{
\includegraphics{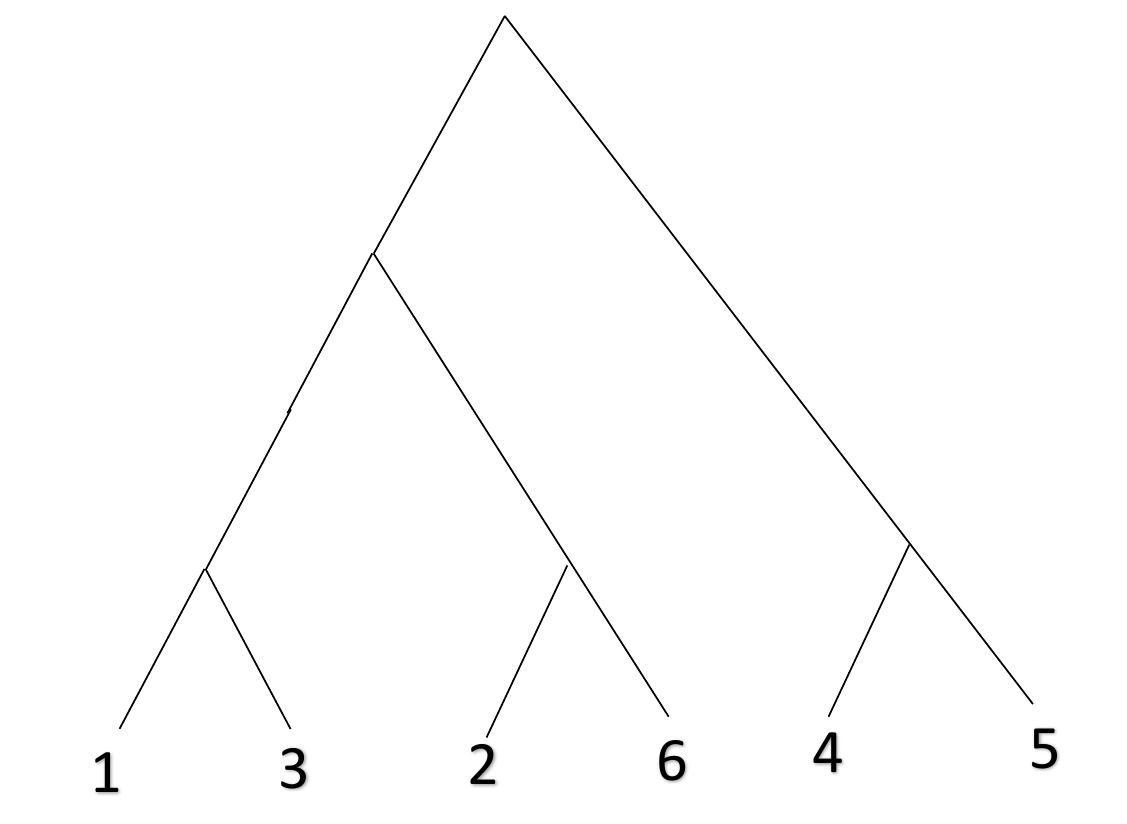}} 

\caption{\label{fig:tree} Two rooted trees $T_1$ and $T_2$ }

\resizebox{!}{4cm}{
\includegraphics{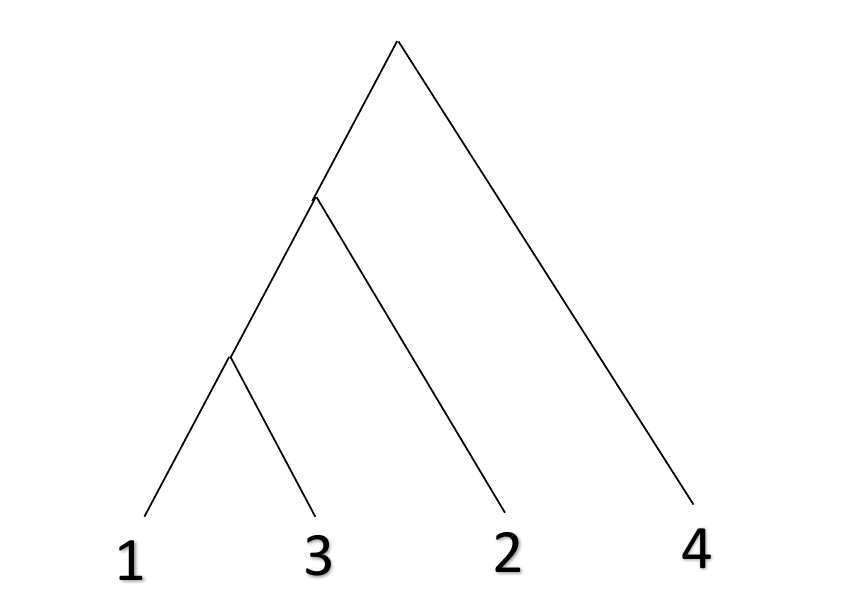}}   
\end{center}
\caption{\label{fig:mast} A maximum agreement subtree for $T_1$ and $T_2$}
\end{figure}

A maximum agreement subtree of a pair of binary trees can be computed in polynomial 
time in $n$ \cite{Steel1993}. 
Let $\mast(T_1,T_2)$ denote the number of leaves of a maximum agreement subtree 
of $T_1 $ and $T_2$. We know from \cite{Martin2013} that if $T_1$ and $T_2$ are any
unrooted binary trees with $n$ leaves, then $\mast(T_1,T_2)=\Omega(\sqrt{\log n})$.  
This contrasts with the rooted case where there can be pairs of rooted trees where
$\mast(T_1,T_2) = 2$.  
Martin and Thatte \cite{Martin2013}  also conjectured that 
if $T_1$ and $T_2$ are balanced rooted  binary
trees with $n$ leaves, then $\mast(T_1, T_2) \geq \sqrt{n}$.  

For the purposes of hypothesis testing,  it is important to understand the distribution
of $\mast(T_1, T_2)$ for trees generated from reasonable distributions
of random trees.     Simulations by Bryant, McKenzie, and Steel \cite{Bryant2003} 
suggest that under the 
uniform and Yule Harding distribution on the rooted binary trees with $n$ leaves, 
the expected size of  $\mast(T_1, T_2)$ is of the order $\Theta(n^a)$ with $a\approx 1/2$. 
It is known that for any sampling consistent and exchangeable distribution on 
rooted binary trees with $n$ leaves (including the uniform and Yule-Harding distributions), 
the expected size of the maximum agreement subtrees 
is less than $\lambda \sqrt{n}$ (for some constant $\lambda >e\sqrt{2}$) \cite{Bernstein2015}.
Lower bounds of order $c n^\alpha$ are also shown in \cite{Bernstein2015} for the 
Yule-Harding and the uniform distribution.
 
In this paper, we study the distribution of $\mast(T_1,T_2)$ where $T_1$ and $T_2$
are trees that are uniformly sampled from all trees with the same shape.  In 
other words, $T_2$ is obtained from $T_1$ by applying a random permutation of
the leaf labels.  In this sense, this gives us a randomized version of 
Martin and Thatte's conjecture for the  case of balanced trees.  We prove that
$\E[\mast(T_1,T_2)] = \Theta(\sqrt{n})$ in this case, which both provides evidence
for Martin and Thatte's conjecture, and provides some further evidence towards the
problems posed in \cite{Bryant2003} for random trees.  Our proof of the
lower bound is based on a structural result about general trees where
we decompose arbitrary trees into substructures we call \emph{blobs}.  The
proof of the upper bound is based on a strengthening of the previously mentioned result
of \cite{Bernstein2015}.  We also show results of simulations that suggest that our ideas
based on blobs could be used to improve lower bounds on the expected value of
$\mast(T_1,T_2)$ for other distributions of random trees.


\section{Lower Bound:  Blobification}
 
In this section we derive a lower bound on the expected size 
of the maximum agreement subtree of two uniformly random trees on $n$ leaves 
with same tree shape. We do this by dividing the trees into what we call as \textit{blobs}, 
which helps us in constructing an agreement subtree between the two trees.

Let $T$ be a rooted binary tree leaf-labeled by $[n]$.
A \emph{cherry blob} is a set of leaves in $T$ consisting of all leaves below a vertex in the
tree. Cherry blobs are also called clades in other phylogenetic contexts.  
An \emph{edge blob} is a nonempty set of leaves of the form $C_1 \setminus C_2$
where $C_1$ and $C_2$ are two nonempty cherry blobs. 
A \emph{blob} in $T$ is either a cherry blob or an edge blob. 

\begin{defn}
Given an integer $k$ and a tree $T$,  a $k$-\emph{blobification} of $T$ is
a collection $\mathcal{B}$ of blobs of $T$ such that,
for all distinct blobs $B_1, B_2 \in \calb$,  $B_1 \cap B_2 = \emptyset$ and
for all $B \in \calb$,  $k \leq |B| \leq 2k-2$.  
\end{defn}

\begin{defn}
Let $T$ be a binary tree, and $\calb$ a $k$-blobification.  
Let $S$ be a set of leaves consisting of one element from each of the blobs in $\calb$.
The \emph{scaffold tree} of the blobification is the unlabelled tree
 $T'$ obtained as the unlabelled version of the induced tree $T|_S$.
\end{defn}

Let $T$ be any rooted binary leaf-labeled tree with $n$ leaves.
We construct a $k$-blobification $\calb$ of $T$ using the following greedy procedure.

First, throw in as many cherry blobs into $\calb$ as possible.
Specifically, among all the cherry blobs $C$ with $k \leq |C| \leq 2k-2$,
we can take the set $\calc$ to consist of all of those cherry blobs
that are minimal, i.e.~that
is, that do not contain any other cherry blobs that have between $k$ and $2k -2$
leaves.

The set of cherry blobs we have constructed $\calc$ induces
a labeled tree that we call the \emph{prescaffold tree}.  This
tree has as leaves all the elements of $\calc$, and can be obtained as an (unlabeled version of
the)
induced subtree $T|_S$ where $S$ is any set of leaves that contain exactly one leaf
from each of the cherry blobs in $\calc$.
If the root of $T|_S$ is not the root of $T$, then we also add an edge onto the
prescaffold tree at the root.  This is illustrated in Figure \ref{fig:scaffold}.
Now we can think about the tree $T$ as consisting of all the leaves 
grouped into blobs of various sizes, each of which attaches somewhere onto the
prescaffold tree.  The leaves that are not part of any of the cherry blobs
will belong to
blobs of size $k -1$ or less that connect onto the prescaffold tree.

On each edge of the prescaffold tree are some number of smaller blobs hanging off
of size $k - 1$ or less.
Working up from the bottom edges of the prescaffold, we can group small blobs together until
they produce an edge blob of size between $k$ and $2k-2$.  This is possible
because each of the small blobs has size $< k$, so when we are grouping blobs together
we have an edge blob with size $< k$ that we add $< k$ more elements to,
we stop when we have formed an edge blob of size between $k$ and $2k -2$.
Let $\cale$ be the resulting set of edge blobs that are produced, that all have size
between $k$ and $2k-2$.
This \emph{greedy $k$-blobification algorithm} stops with a blobification  
$\calb = \calc \cup \cale$  where on each edge of the scaffold tree
there are leftover small blobs whose total number of leftover leaves is at most $k-1$.
The set $\calb = \calc \cup \cale$ is called the \emph{greedy $k$-blobification}.

Starting with the prescaffold tree $T'$ and adding a leaf attached to an edge for
each time an edge blob gets formed, we arrive at an unlabelled tree we call the
scaffold tree.

\begin{ex}
Consider the binary tree on $17$ leaves pictured in Figure \ref{fig:blb}.
We first consider the greedy $2$-blobification.  Note that the cherry
blobs are exactly the cherries in this case.  These are the sets
$\{1,2\}, \{7,8\}, \{11,12\}, \{13, 14\}$.  The prescaffold tree
is shown on the left of Figure \ref{fig:scaffold}.
Note that there is an edge that hangs off the root.
The edge blobs in this example are $\{3,4\}, \{5,6\}, \{15, 16\}$.
The resulting scaffold tree is the tree on the right in Figure 
\ref{fig:scaffold}.  Note that leaves $9$, $10$, and $17$ do not
end up in any blob.

On the other hand, consider the greedy $3$-blobification of the same
tree.  There are two cherry blobs, $\{1,2,3\}$ and $\{11,12,13,14\}$.
The edge blobs are $\{4,5,6\}$, $\{7,8,9\}$,  and $\{15,16,17\}$.

\begin{figure} 
\begin{center}
\resizebox{!}{5.5cm}{
\includegraphics{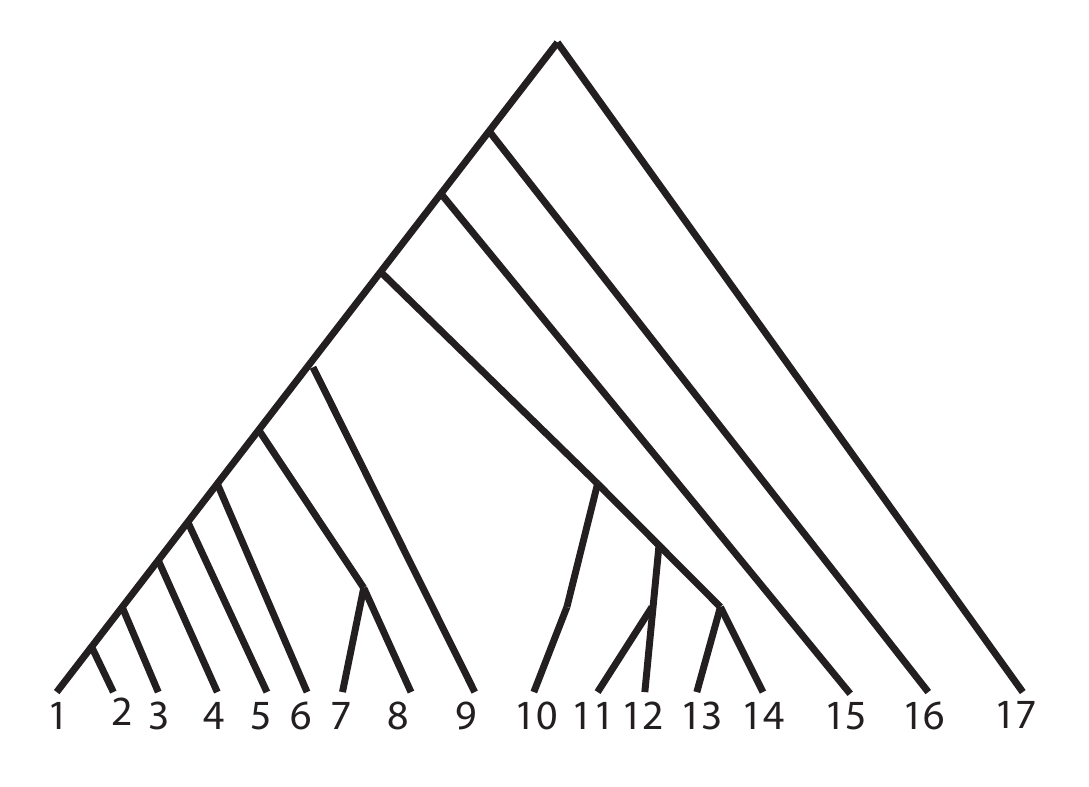}}   
\end{center}
\caption{\label{fig:blb} A tree}
\end{figure}

\begin{figure} 
\begin{center}
  \resizebox{!}{3.5cm}{
\includegraphics{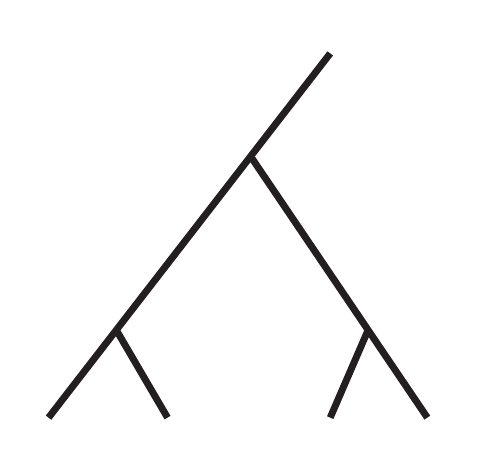}} \quad \quad   \resizebox{!}{3.5cm}{
\includegraphics{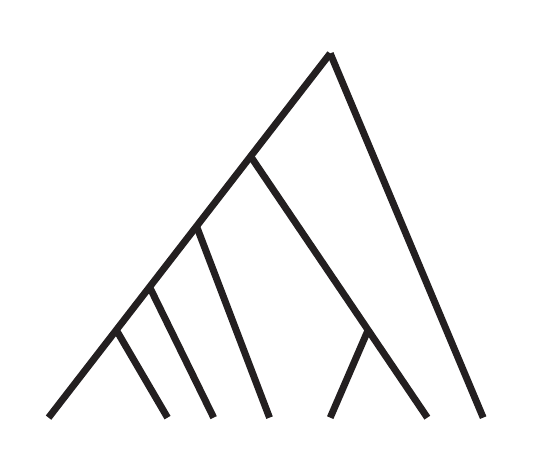}} 
\end{center}
\caption{\label{fig:scaffold} Prescaffold and scaffold tree for the 2-blobification}
\end{figure}

\end{ex}

\begin{prop}\label{prop:blob}
Let $T$ be a rooted binary leaf-labeled tree with $n$ leaves.  Then $T$ has
a $k$-blobification with at least $\frac{n}{4k}$ blobs.
\end{prop}

\begin{proof}

We apply the \textit{k-blobification} algorithm on $T$. Let the final collection $\calb$ of blobs contain $a$ cherry blobs and
$b$ edge blobs.  Since the prescaffold tree is a binary rooted tree with $a$ leaves, there
are at most $2a - 1$ edges (potentially there is a root edge).  
Taking everything at its most extreme, we see that the total
number of leaves, $n$ is at most
\[
n  \leq  (a + b)(2k-2)  + (k-1)(2a - 1)  =  (4a + 2b -1)(k-1) \leq (4a + 4b)k
\] 
where the first part comes from the contribution from each of the $a+b$ blobs, and the second term
is the leftover leaves.
The total number of blobs is $a + b$, which is greater than $n/4k$ from the above inequality.
\end{proof}

\begin{lemma}\label{lemma:birthday}
Let $S_1$ and $S_2$ be uniformly random subsets of $[n]$, each of size at least $\sqrt{n}$.
The probability that $S_1 \cap S_2 \neq \emptyset$ is at least $1-e^{-1}$.
\end{lemma}

\begin{proof}
The probability that $S_1 \cap S_2 \neq \emptyset$ is clearly minimized when
both $S_1$ and $S_2$ have $\sqrt{n}$ elements.  In this case, the probability that
$S_1 \cap S_2 = \emptyset$
is given by the formula
\begin{eqnarray*}
\frac{ \binom{n - \sqrt{n}}{\sqrt{n}}}{ \binom{n}{\sqrt{n}}} & = &  \prod_{i = 1}^{\sqrt{n}} 
\left( 1 - \tfrac{\sqrt{n}}{n-i+1}\right)  \\
&  \leq  & \left(1 - \tfrac{1}{\sqrt{n}} \right)^{\sqrt{n}}  \\
&  \leq  &  e^{-1}
\end{eqnarray*}
This shows that the probability that $S_1 \cap S_2 \neq \emptyset$ is
at least $1-e^{-1}$.
\end{proof}

\begin{thm}\label{thm:blobn12}
Let $T_1$ and $T_2$ be two uniformly random trees on $n$ leaves among all trees with the
same tree shape (i.e. $T_2$ is a random leaf relabeling of $T_1$).
Then the expected size of $\mast(T_1,T_2)$ is at least $\sqrt{n}(1 - e^{-1})/4 $.
\end{thm}

\begin{proof}
Consider the $\sqrt{n}$-blobification of $T_1$ and $T_2$,
which we denote by $\calb_1$ and $\calb_2$.  Since the trees have the same tree shape, 
this blobification has the same scaffold tree $T'$.  We can order the blobs in $\calb_1 = \{B_{11}, \ldots, B_{1s} \}$
and $\calb_2 =  \{B_{21}, \ldots, B_{2s} \}$ so that
$B_{1i}$ and $B_{2i}$ correspond to the same leaf in the scaffold tree $T'$.

If for each $i$, we had that $B_{1i} \cap B_{2i} \neq \emptyset$, we could take one
leaf $\ell_i \in B_{1i} \cap B_{2i}$, and let $S = \{\ell_1, \ldots, \ell_s\}$, we would
have $T_1|_S  = T_2|_S$ and this common agreement subtree would have the same shape
as the scaffold tree $T'$.  

Note that, since our trees are uniformly random among
all trees with a given fixed shape, the probability that $B_{1i} \cap B_{2i} \neq \emptyset$ is at least $1-e^{-1}$
by Lemma \ref{lemma:birthday}, so that the expected number of $i$ where 
$B_{1i} \cap B_{2i} \neq \emptyset$ is at least $s(1 - e^{-1})$.  This set of index
positions gives an agreement subtree of expected size at least $s(1 - e^{-1})$,
which will be isomorphic to an induced subtree of the scaffold tree.
Since $s \geq \sqrt{n}/4$ by Proposition \ref{prop:blob} we see that
the expected size of $\mast(T_1,T_2)$ is at least $\sqrt{n}(1 - e^{-1})/4 $.
\end{proof}

The same argument can be used to show that if $T_1$ and $T_2$ are uniformly
random trees among all trees that have the same $\sqrt{n}$-blobification, the expected 
value of $\mast(T_1,T_2)$ will also be at least $\sqrt{n}(1 - e^{-1})/4 $.


\section{Upper bound:  Eliminating Sampling Consistency}

In this section we generalize the result obtained from \cite{Bernstein2015} 
that if $T_1$ and $T_2$ are generated from any sampling consistent and 
exchangeable distribution on rooted binary trees with $n$ leaves,
the expected size of the $\mast$ is less than $\lambda \sqrt{n}$ (for some constant $\lambda >e\sqrt{2}$). 
We show that the result holds true even if we remove sampling consistency as one of the conditions. 
Since the distribution of random trees with the same shape is exchangeable, this will
prove an $O(\sqrt{n})$ bound on the expected size of the maximum agreement subtree
for uniformly random trees with the same shape.

Let $RB(n)$ denote the set of all rooted binary trees with $n$ leaves. 
For a set $S$ let $RB(S)$ denote the set of all rooted binary trees with leaf label
set $S$.

\begin{defn}
A distribution on $RB(n)$ is said to be \textit{exchangeable} 
if any two trees which differ only by a permutation of leaves have the same probability.
\end{defn}

For each $n = 1, 2, \ldots, $ we can consider a probability distribution $P_n$ on $RB(n)$.
We denote the probability of a tree $t \in RB(n)$ by $P_n[t]$.  
The notion of sampling consistency is concerned with a probability model
for random trees that describes probability distributions for random trees for all $n$.
For example, the uniform distribution on trees gives a probability distribution $P_n$
for each $n$, where $P_n[t] =  \tfrac{1}{(2n-3)!!}$ for all $t \in RB(n)$.
The property of sampling consistency is one that concerns the entire family
of probability distributions $P_n$, $n = 1, 2, \ldots$.

\begin{defn}
A distribution of random trees is said to satisfy \textit{sampling consistency} 
if for all $n$,  all  $s < n$, all $S \subseteq [n]$ with $|S| = s$, and all $t \in RB(S)$,
\[
P_s[t]  =   \sum_{T \in RB(n) : T|_{S} = t}  P_n[T].
\]
\end{defn}

In other words, in a sampling consistent distribution  if we take a random tree $T$
and restrict to a random subset of the leaves, the resulting tree has the same distribution
as if we had just chosen a random tree on that subset of leaves, directly.  
Our goal in this section is to remove the restriction of sampling consistency for
the following theorem from \cite{Bernstein2015}.

\begin{thm}\label{thm:bern}
Consider an 
exchangeable and sampling consistent distribution on rooted binary trees. 
Then for any $\lambda > e\sqrt{2}$ there is a value $m$ such that, 
for all $n\geq m$, 
\[
\E[\mast(T_1,T_2)]\leq \lambda \sqrt{n}
\]
where $T_1, T_2$ are sampled from this distribution. 
\end{thm}

Let $P_n$ be an exchangeable distribution on $RB(n)$.  Since we only consider a fixed
value of $n$, we do not have sampling consistency.
To prove an analogue of Theorem \ref{thm:bern} without sampling consistency depends
on defining some new probability distributions on $RB(s)$ for $s < n$.  
Specifically, for any $s<n$, and $t \in RB(s)$ we define 
\[
P_s[t]=  \sum_{T \in RB(n) : T|_{[s]} = t}  P_n[T].
\]
We can also use the notation $P_s[t]=  P_n[ T|_{[s]} = t]$ to denote this same probability.

\begin{prop}\label{prop:main}
Let $P_n$ be an exchangeable distribution defined on $RB(n)$. Then for any $s<n$, 
$P_s$ satisfies exchangeability property on $RB(s)$. 
\end{prop}

\begin{proof}
Let $t$ and $ t'$ be two trees in $RB(s)$ with same tree shape, and let $s < n$.  
By definition, $P_s[t]=P_n[T|_{[s]}=t]$ and 
$P_s[t']=P_n[T|_{[s]}=t']$. We define a bijection $\phi:[s]\rightarrow [s] $ from $[s]$ 
to itself such that $\phi(t)=t'$ and extend the map $\phi:[n]\rightarrow [n]$ from $[n]$ to itself
with $\phi(a)=a$, for all $a>s$.

So, for any two trees $T,T'$ in  $RB(n)$ with $T|_{[s]}=t$ and $T'|_{[s]}=t'$, 
we have
\[
\phi(T)|_{[s]}=\phi(t)=t'\text{ and } \phi^{-1}(T)|_{\{1,2,...,s\}}=\phi^{-1}(t')=t.
\]
Hence $T|_{[s]}=t$ if and only if $\phi(T)|_{[s]}=t'$ since
any bijection from $[n]$ to $[n]$ induces a bijection from $RB(n)$ to $RB(n)$.
Also, as $T$ and $\phi(T)$ have the same tree shape and $P_n$ is exchangeable, we have $P_n[T]=P_n[\phi(T)].$
Hence we can conclude that $P_s[t]=P_s[t']$.
\end{proof}

\begin{lemma}
Suppose that phylogenetic trees $T_1$ and $T_2$ in $RB(n)$
are randomly generated under a model that satisfies exchangeability. 
Then 
\[
P[MAST(T_1,T_2)\geq s]\leq \psi_{n,s}={n\choose s} \sum_{t\in RB(s)} P_s[t]^2,
\]
where $P_s[t]$ is defined as $P_s[t]=P_n[T|_{[s]}=t]$ for $t\in{RB(s)}$.
\end{lemma}

\begin{proof}
This theorem can be proved exactly the way Lemma 4.1 of \cite{Bryant2003}
is proved with the last equality following from the way we have defined 
$P_s[t]$ instead of using sampling consistency.  The details are included here
for completeness.

Given a subset $S$ of $[n]$ let 
\[
X_S=   \left\{
\begin{array}{ll}
      1, & \text{if } T_1|_S=T_2|_S \\
      0, & \text{otherwise.} \\
\end{array} 
\right. \]
The number of agreement subtrees with $s$ leaves for $T_1$ and $T_2$ is counted by
\[X^{(s)}=\sum_{S\subseteq [n]:|S|=s}X_S.\]
The event ${MAST(T_1,T_2)\geq s}$ is equivalent to the event ${X^{(s)}\geq 1}$, so 
\begin{eqnarray*}
P[MAST(T_1,T_2)\geq s] & = & P[X^{(s)}\geq 1]  \\
&  \leq  &  E[X^{(s)}]  \\
&  =  &  \sum_{S \subseteq[n]: |S|=s} E[X_S] \\
& =  &\sum_{S \subseteq[n]: |S|=s} P[X_S=1]  \\
&  =  & {n\choose s}P[X_{[s]}=1], 
\end{eqnarray*}
where the last equality is by exchangeability. Now, 
\begin{eqnarray*}
P[X_{[s]}=1] & = & P_n[T_1|_{[s]}=T_2|_{[s]}]\\
& = & \sum_{t\in RB(s)} P_n[T_1|_{[s]}=t \text{ and } T_2|_{[s]}=t] \\
& = & \sum_{t\in RB(s)} P_n[T_1|_{[s]}=t]^2  \\
& = & \sum_{t\in RB(s)} P_s[t]^2 
\end{eqnarray*}
where the last equality follows from the way we have defined $P_s[t]$. Upon substituting back for this term, we obtain the upper bound as stated in the lemma.
\end{proof}

We now state a proposition from \cite{Bernstein2015}.

\begin{prop}
\cite[Proposition 4.2]{Bernstein2015} 
Let $P_s$ be any exchangeable distribution on rooted binary trees with $s$ leaves. 
Then 
\[ 
\sum_{t\in RB(s)} P_s(t)^2\leq \frac{2^{s-1}}{s!}.
\] 
\end{prop}

Now we can combine these results to deduce the strengthened version of 
Theorem \ref{thm:bern} that does not require sampling consistency.

\begin{thm}\label{thm:last}
 Then for any $\lambda > e\sqrt{2}$ there is a 
value $m$ such that, for all $n\geq m$, 
\[
\E[\mast(T_1,T_2)]\leq \lambda \sqrt{n}.
\]
where $T_1$ and $T_2$ are sampled from any exchangeable distribution on $RB(n)$.
\end{thm}

\begin{proof}
This theorem can be proved exactly the way as Theorem 4.3 
in \cite{Bernstein2015} is proved as we have already shown that $P_s$ is exchangeable 
by Proposition \ref{prop:main}.

We explore the asymptotic behaviour of the quantity $\phi_{n,s}={n\choose s} \frac{2^{s-1}}{s!}$. Using the inequality ${n\choose s}\leq \frac{n^s}{s!}$ and Stirling's approximation, we have: 
$$\phi_{n,s}\leq \frac{1}{4\pi s}\left(\frac{2e^2 n}{s^2}\right)^s\theta(s) $$
where $\theta(s)\sim 1$. Hence, $\phi_{n,s}$ tends to zero as an exponenential function of $n$ as $n\rightarrow \infty$. Since $\phi_{n,s}\geq \psi_{n,s}$, we see that $P[\mast(T_1,T_2)\geq \lambda \sqrt{n}]$ tends to zero as an exponential function of $n$. Since $\mast(T_1,T_2)\leq n$, this implies that $E[\mast(T_1,T_2)]\leq \lambda \sqrt{n}$. 
\end{proof}

Now we can deduce the main result for trees with the same shape.

\begin{cor}\label{cor:generalupper}
Let $T_1$ and $T_2$ be generated from the uniform distribution on
rooted binary trees  with $n$ leaves with same tree shape (that is, $T_2$ is a random leaf relabeling of $T_1$).
Then for any $\lambda > e\sqrt{2}$ there is a value $m$ such that, for all $n\geq m$, 
$$E[\mast(T_1,T_2)]\leq \lambda \sqrt{n}.$$
\end{cor}

\begin{proof}
This follows immediately from Theorem \ref{thm:last} since the uniform distribution on trees with the same
shape is exchangeable.  
\end{proof}

Combining Theorem \ref{thm:blobn12} and Corollary \ref{cor:generalupper} we deduce
the main result of the paper.

\begin{thm}
Let $T_1$ and $T_2$ be generated from the uniform distribution on
rooted binary trees  with $n$ leaves with same tree shape 
(that is, $T_2$ is a random leaf relabeling of $T_1$).
Then
\[
E[\mast(T_1,T_2)] =  \Theta(\sqrt{n}).
\]
\end{thm}


\section{Simulations with Blobification}

The blobification idea has the potential to be useful for proving
lower bounds on the expected size of the maximum agreement subtree in other
contexts.  For example, suppose we have a model for random trees on $n$ leaves
and we can show that the scaffold tree of the $\sqrt{n}$-blobification of a random 
tree has depth $\geq f(n)  $ with high probability $p>0$ that does not depend on $n$.
Then under this model, using Lemma \ref{lemma:birthday}, we see that
two random trees will have an agreement subtree of expected size at least $f(n)(1 - e^{-1})p^2$.
Such a tree would be obtained as a comb tree by comparing blobs that are matched along the
path from the root to the deepest leaf in each scaffold tree.  
Hence, understanding the distribution of the depth of the scaffold trees
in the $\sqrt{n}$-blobification could give improved lower bounds on the 
expected size of the maximum agreement subtree in some random tree models.

One specific application where this perspective might prove useful is for
uniformly randomly trees.  The current best lower bound for the expected size
of the maximum agreement subtree for two uniformly random trees on $n$ leaves
is $\Omega(n^{1/8})$ \cite{Bernstein2015}.  To see if this blobification idea
might be useful for improving the lower bound, we simulated a lower bound for the
depth of the scaffold tree of a uniformly random tree using the following greedy procedure.  

\begin{alg}[Greedy Comb Scaffold] ${} $  

Input:  A binary tree $T$ and an integer $k$.

Output:  A scaffold tree in shape of a comb, whose leaves correspond to blobs of size $\geq k$.
\begin{itemize}
\item  Set $u =  ()$.
\item  While $T$ has more than one leaf Do:
\begin{itemize}
\item  Let $T_1$ and $T_2$ be the left and right subtrees of the root in $T$.
\item  Append $\min(\#(T_1), \#(T_2))$ to $u$. 
\item  Set $T$ equal to the larger of $T_1$ and $T_2$.
\end{itemize} 
\item  Set $v = (0)$.
\item  While $u \neq ()$ do
\begin{itemize}
\item  If the last element of $v$ is greater than or equal to $k$, append the last element of $u$ to $v$.
\begin{itemize}
\item  Else, add the last element of $u$ to the last element of $v$.
\end{itemize}
\item  Delete the last element of $u$.  
\end{itemize}
\item Output $v$, a vector of sizes of blobs in $T$, all except the last one having size $\geq k$,
which have a scaffold that is a comb tree.
\end{itemize}
\end{alg}

Note that the length of the vector $v$ (or possibly the length minus $1$) gives the  
number of leaves in the greedy comb scaffold where all blobs will have size greater than $k$. 

\begin{figure} 
\begin{center}
  \resizebox{!}{5.5cm}{
\includegraphics{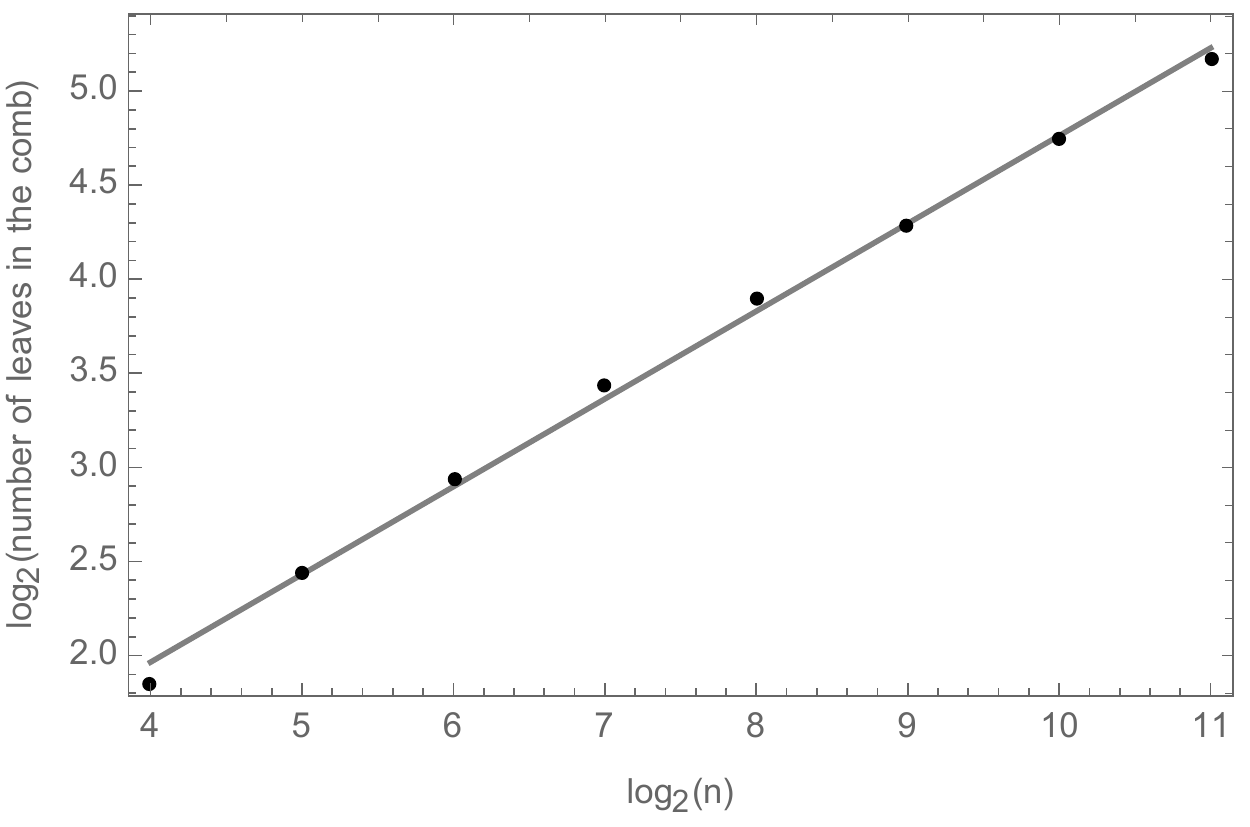}} 
\end{center}
\caption{\label{fig:fitline} Log-log plot of the simulated expected size of the greedy comb scaffold}
\end{figure}

We applied the greedy comb scaffold algorithm to uniformly random binary trees with $k = \sqrt{n}$
on $2^n$ leaves for $n = 4, \ldots, 11$, with $1000$ samples for each value of $n$.
The results of these simulations are displayed in the log-log plot of Figure \ref{fig:fitline}.    
The slope of the line of best fit is approximately $.466$. These data suggest 
that a strategy based on  blobification  could yield an $\Omega(n^{.466})$ lower
bound on the size of the maximum agreement subtree for uniformly random trees. This would be
a significant 
improvement on our estimates of the expected size of the maximal agreement 
subtree for uniformly random trees, given the current best known lower bound of $\Omega(n^{1/8})$.


\end{document}